\documentclass[letterpaper, 10 pt, conference]{ieeeconf}
\IEEEoverridecommandlockouts
\usepackage{cite}
\usepackage{amsmath,amssymb,amsfonts}

\usepackage{graphicx}
\usepackage{textcomp}
\usepackage{xcolor}
\usepackage{mathrsfs}
\usepackage{bm}
\usepackage{algorithm}
\usepackage{algpseudocode}
\usepackage{caption}
\usepackage{subcaption}
\usepackage{stfloats}
\usepackage{wrapfig}

\DeclareMathOperator*{\argmin}{arg\,min}

\newtheorem{Theorem}{Theorem}

\newtheorem{Corollary}{Corollary}
\newtheorem{remark}{Remark}
\title{
Convexifying Mean-Field Control: An Occupation-Measure and Frank–Wolfe Approach
}

\author{ Di Yu, Sixiong You and Chaoying Pei
	\thanks{Di Yu is with the Department of Statistics, Purdue University, West Lafayette, IN 47907, USA. Email: {\tt\small yu1128@purdue.edu}}
    \thanks{Sixiong You is with Eli Lilly and Company, Indianapolis, IN 46225, USA. Email: {\tt\small yousixiong@gmail.com}}
	\thanks{Chaoying Pei is with the Department of Mechanical and Aerospace Engineering, Missouri University of Science and Technology, Rolla, MO 65401, USA. Email: {\tt\small cpk4t@mst.edu}}
}

\begin{document}

\maketitle

\begin{abstract}
    Large-scale robotic swarms motivate the use of mean-field control (MFC). Classical partial differential equation (PDE)-based formulations provide a principled framework but {can become computationally challenging in higher dimensions}, whereas machine learning achieves scalability at the cost of approximation and guarantees. {In this work, we establish an optimization-based framework that lifts the MFC problem into the space of occupation measures, resulting in a convex relaxation formulated as an optimization over measures.} The resulting problem is solved using a Frank–Wolfe (FW) algorithm in the measure space, with each iteration reduced to a tractable optimal control problem. This approach retains the O(1/k) convergence rate of FW, avoids discretization of the state space, and naturally incorporates interaction and safety constraints. Numerical experiments demonstrate agreement with analytic and PDE-based baselines in two dimensions and show that the method scales to three-dimensional environments with multiple obstacles, {where standard grid-based PDE solvers become impractical.} A full 3D instance with ten obstacles is solved in minutes on a standard workstation, underscoring the practicality and scalability of the proposed framework.
\end{abstract}

\section{Introduction}
In recent years, robotic systems such as unmanned aerial vehicles (UAVs) have increasingly moved toward swarm deployments, where large populations of agents operate in a coordinated manner. As the swarm size grows, traditional multi-agent control techniques become inadequate due to scalability limits and the difficulty of enforcing interactions across thousands of agents. This challenge has motivated the development of mean-field control (MFC), which provides a principled framework for modeling and optimizing collective behaviors in large populations \cite{elamvazhuthi2019mean}. 

MFC formulates swarm coordination by treating the time-varying state distribution of agents as a probability measure and optimizing its evolution under prescribed objectives such as performance, interaction, and safety. Research in MFC has progressed along several directions. Macroscopic models based on partial differential equations (PDE), or stochastic processes capture swarm behavior at the distribution level and have been applied to coverage, task allocation, and consensus, though PDE-based formulations can face significant scalability challenges as the state dimension grows \cite{fornasier2014mean}. Geometric approaches connect MFC with optimal transport, showing that swarm tracking can be formulated in Wasserstein space, where optimal trajectories follow geodesics \cite{emerick2023continuum}. In parallel, {machine learning (ML) methods}—particularly deep reinforcement learning \cite{chen2020mean} and neural approximations \cite{ruthotto2020machine} of value functions—have become popular for high-dimensional settings, achieving scalability but lacking rigorous guarantees. 

{To alleviate the limitations of PDE- and learning-based approaches, recent work has increasingly focused on optimization-based frameworks for MFC.} Fung and Nurbekyan introduced MFC barrier functions to enforce safety with provable guarantees \cite{fung2025mean}, while Vidal and collaborators employed {kernel expansion and primal–dual formulation} to retain tractability and convergence \cite{vidal2025kernel}. {In parallel, learning-based solution schemes such as fictitious-play and best-response dynamics have been widely studied in the mean-field game literature as iterative methods for computing equilibria, with convergence guarantees under suitable assumptions \cite{cardaliaguet2017learning}.} Together, these developments outline a field where PDE methods provide theoretical foundations, ML dominates practice, and convex optimization emerges as a rigorous and scalable alternative. These advances begin to narrow the gap between rigorous PDE models and scalable ML approaches, but the inherent nonconvexity of many MFC formulations continues to motivate new methods with stronger guarantees. {However, many of these approaches still operate at the PDE or equilibrium level and may face scalability challenges in high-dimensional settings.} Therefore, this motivates approaches that combine scalability with theoretical guarantees.

To address this need, we leverage occupation measures, an established convexification tool in optimal control~\cite{lasserre2008nonlinear}, to reformulate the mean–field control problem as an optimization over measures with linear dynamical constraints, leading to problem~\eqref{mainObj}. This reformulation transforms the inherently nonconvex MFC problem into an exact convex program over measures.

Optimization over measures has been studied extensively in statistics, signal processing, and ML~\cite{1960kie,2017boygeorec,2018mei}. As an infinite–dimensional problem, it is often approximated in finite dimensions, which sacrifices accuracy and discards structural information. Moreover, for grid-based discretization approaches, the computational cost scales exponentially with grid resolution due to the curse of dimensionality~\cite{2017boygeorec,yu2025deterministic}. To overcome these issues, we solve the infinite–dimensional problem directly in measure space without discretization. In problem~\eqref{mainObj}, the first variation of the objective admits a closed form, which enables efficient application of first–order algorithms such as Frank–Wolfe (FW) with provable convergence guarantees.

{Recent work has rigorously established that Frank–Wolfe methods can be applied to optimization over measures in infinite-dimensional settings, with provable convergence guarantees~\cite{yu2025deterministic,yu2025frank,yu2025derivative}.} These methods are projection--free, making them naturally suited to handle linear dynamical constraints, and they avoid discretization by operating directly in the infinite--dimensional measure space. {Related generalized conditional-gradient methods have also been studied in infinite-dimensional and mean-field settings, particularly at the PDE or potential-game level \cite{lavigne2023generalized}.} Complementary to conditional-gradient approaches in the mean-field game literature that operate at the PDE level, {our occupation-measure formulation enables problem~\eqref{mainObj} to be solved iteratively via FW through a sequence of classical optimal control subproblems.} These subproblems can be efficiently solved using a wide range of existing optimal control techniques, thereby opening a new and scalable pathway for solving MFC while preserving feasibility and convexity, with standard $O(1/k)$ convergence guarantees and closed-form trajectory-based updates.

The main contributions of this work are as follows:
\begin{itemize}
    \item {\textbf{Convex formulation via occupation measures:}
We show that the MFC problem admits a convex relaxation in the occupation-measure space, avoiding explicit state-space discretization and providing a principled foundation for large-scale swarm coordination.}
    \item \textbf{Iterative solvability:} We design a Frank--Wolfe algorithm that operates directly in the infinite-dimensional measure space, transforming the problem into a sequence of classical optimal control subproblems. These subproblems can be efficiently solved with existing tools, yielding a simple yet powerful method with standard $O(1/k)$ convergence guarantees. 
    \item \textbf{Scalable validation:} We demonstrate through extensive experiments that our method matches analytic and PDE-based baselines in two dimensions, and crucially, scales to three-dimensional environments with multiple obstacles. Full 3D scenarios with ten obstacles are solved in minutes on a standard workstation---well beyond the reach of PDE solvers---highlighting the practicality and scalability of the framework. 
\end{itemize}

The remainder of the paper is organized as follows: Section~\ref{sec:problem} presents the problem statement and the occupation–measure reformulation. Section~\ref{sec:properties} establishes key properties of the resulting measure optimization problem. Section~\ref{sec:fw_methods} develops Frank–Wolfe methods tailored to this setting. Section~\ref{sec:numerical} provides numerical illustrations in both two- and three-dimensional environments. Section~\ref{sec:conclusion} concludes the paper.

\section{Problem Statement and Reformulation}\label{sec:problem}
In this section, we introduce the $N$-agent swarm model, formulate the associated cost functional, and present its reformulation in the occupation–measure space.
\subsection{$N$-Agent model and Cost Functional}
We consider a swarm of $N$ unmanned aerial vehicles (UAVs) operating over a fixed horizon $[0,T]$. Each agent follows nonlinear dynamics and seeks to coordinate with the swarm under performance, interaction, and safety objectives. The state of the $i$-th agent is $x_i(t)\in\mathbb{R}^{n_x}$ and the control input is $u_i(t)\in\mathbb{R}^{n_u}$, evolving as
\begin{equation}
    \dot{x}_i(t) = f(x_i(t),u_i(t)), 
    i=1,\dots,N,\; t \in [0,T].
\end{equation}

Consider $N$ identical UAVs with states $x_i(t)$ and controls $u_i(t)$. 
The empirical distribution is $\rho_t^N := 1/N \sum_{i=1}^N \delta_{x_i(t)}$. The $N$-agent cost is
\begin{align}\label{eq:NagentMF}
J_N &= \frac{1}{N}\sum_{i=1}^N \int_0^T \Big( \ell_0(t,x_i(t),u_i(t))
+ \lambda (W*\rho_t^N)(x_i(t))\Big) dt \nonumber\\
&\qquad+ \int \Psi(x)\,d\rho_T^N(x).
\end{align}
Here $\ell_0(t,x,u)$ denotes the individual running cost (e.g., control effort or deviation from a desired trajectory), $W:\mathbb{R}^{n_x}\to\mathbb{R}$ is an interaction potential modeling pairwise effects such as collision avoidance or cohesion, $\lambda>0$ is a scalar weighting the interaction term, and $\Psi(x)$ is a terminal cost function encoding the swarm objective at the final time (e.g., reaching a target region). This type of $N$-agent cost is standard in MFC, see e.g., \cite{carmona2018probabilistic}. We assume all agents share the same initial condition $x_i(0)=x_0$, so that the empirical initial distribution reduces to $\rho_0^N = \delta_{x_0}$.


\subsection{Occupation Measures for Single Trajectories}

To prepare for the reformulation of the $N$-agent cost, we associate each
trajectory--control pair with measures on time, state, and control space. 
Let $\mathcal{X}\subset\mathbb R^{n_x}$ and $\mathcal{U}\subset\mathbb R^{n_u}$ 
be compact sets, and define $\Sigma := [0,T]\times \mathcal{X}\times \mathcal{U}$. Denote by $\mathcal M(\Sigma)$ the Banach space of finite signed Borel measures 
on $\Sigma$, equipped with the total variation norm, and by 
$\mathcal M_+(\Sigma)$ its positive cone. 
We focus on the subset
\[
\mathcal M_+(\Sigma,T):=\{\mu\in\mathcal M_+(\Sigma):\ \mu(\Sigma)=T\},
\]
which is convex. When $T=1$, $\mathcal M_+(\Sigma,1)$ coincides with the set of probability measures on $(\Sigma,\mathcal B(\Sigma))$, {where $\mathcal B(\cdot)$ denotes the Borel $\sigma$-algebra}. Similarly, we write $\mathcal M_+(\mathcal{X},1)$ for the set of probability 
measures on $(\mathcal{X},\mathcal B(\mathcal{X}))$.

Given an admissible trajectory
{\small
\begin{equation*}
\tau := \{(x(t),u(t)) : t\in[0,T],\; \dot{x}(t)=f(x(t),u(t)), \; x(0)=x_0\}
\end{equation*}
}
we define the associated \emph{occupation measures} 
$(\mu[\tau],\nu[\tau])$ as follows (see e.g., \cite{lasserre2008nonlinear}). The \emph{running occupation measure} $\mu[\tau]\in\mathcal{M}_+(\Sigma,T)$ is
\[
\mu[\tau](A\times B\times C) 
:= \int_A \mathbf{1}_{B\times C}(x(t),u(t))\,dt,
\]
for any $A\in\mathcal{B}([0,T])$, $B\in\mathcal{B}(\mathcal{X})$, and 
$C\in\mathcal{B}(\mathcal{U})$. 
Equivalently, for any $\varphi\in C(\Sigma)$,
\begin{equation}\label{eq:runningMeasure}
\int_\Sigma \varphi(t,x,u)\,d\mu[\tau](t,x,u) 
= \int_0^T \varphi(t,x(t),u(t))\,dt.
\end{equation}
By construction, $\mu[\tau](\Sigma)=T$. 

The \emph{terminal occupation measure} 
$\nu[\tau]\in \mathcal{M}_+(\mathcal{X},1)$ is
\[
\nu[\tau](B) := \mathbf{1}_B(x(T)), \qquad \forall B\in \mathcal{B}(\mathcal{X}),
\]
that is, $\nu[\tau]=\delta_{x(T)}$. Equivalently, for any $\psi\in C(\mathcal{X})$,
\begin{equation}\label{eq:terminalMeasure}
\int_\mathcal{X} \psi(x)\,d\nu[\tau](x) = \psi(x(T)).
\end{equation}

Thus the pair $(\mu[\tau],\nu[\tau])$ encodes the entire trajectory $\tau$ in measure form. Later we will form empirical averages of these single-trajectory measures to represent the whole swarm, which yields a measure-based reformulation of the $N$-agent cost. Occupation measures have a long history in control theory as a convexification tool; see \cite{lasserre2008nonlinear} for detailed expositions.

\subsection{Reformulation via Occupation Measures}
The next theorem provides the key link between the classical $N$-agent cost and its occupation–measure representation.
\begin{Theorem}[Occupation–measure form of $J_N$]\label{Thm:objJN}
For agent $i=1,\dots,N$ with trajectory $\tau_i$ and occupation measures
$(\mu[\tau_i],\nu[\tau_i])$, define $\mu^N:=\tfrac1N\sum_{i=1}^N \mu[\tau_i]\in\mathcal M_+(\Sigma,T)$ and $\nu^N:=\tfrac1N\sum_{i=1}^N \nu[\tau_i]\in\mathcal M_+(\mathcal X,1)$. Assume $\ell_0,\Psi$ are bounded continuous and $W$ is measurable with finite
integrals. Then the $N$-agent cost \eqref{eq:NagentMF} can be rewritten as
\begin{align}\label{eq:JN-measure}
J_N = \int_{\Sigma} \ell_0(t,x,u)\,d\mu^N(t,x,u)+ \int_{\mathcal X}\Psi(x)\,d\nu^N(x) \nonumber\\
+ \lambda\!\!\iint_{\Sigma\times\Sigma}\!\! W(x-y)\,\delta(t-t')\,d\mu^N(t,x,u)\,d\mu^N(t',y,v).
\end{align}
Here $\delta(t-t')$ denotes the Dirac distribution enforcing $t=t'$ in the
interaction term.

Moreover, the measures $(\mu^N,\nu^N)$ satisfy the averaged weak Liouville identity:
for all $v\in C^1([0,T]\times\mathcal X)$,
\begin{equation}\label{eq:weakN}
\int_{\mathcal X} v(T,x)\,d\nu^N(x) - v(0,x_0) = \int_{\Sigma}\big(\partial_t v+\nabla_x v\!\cdot\! f(x,u)\big)\,d\mu^N.
\end{equation}
\end{Theorem}
\begin{proof} Denote the three parts of $J_N$ in~\eqref{eq:NagentMF} by
\begin{align*}
    I&:=\frac{1}{N}\sum_{i=1}^N \int_0^T \ell_0(t,x_i(t),u_i(t)) dt; \\ 
    II&:=\frac{\lambda}{N}\sum_{i=1}^N \int_0^T (W*\rho_t^N)(x_i(t)) dt;\;\;III&:= \int \Psi(x)\,d\rho_T^N(x).
\end{align*}
From~\eqref{eq:runningMeasure}, for each agent $i$,
\[
\int_0^T \ell_0(t,x_i(t),u_i(t))\,dt
= \int_{\Sigma} \ell_0(t,x,u)\,d\mu[\tau_i](t,x,u).
\]
Averaging over $i$ yields
\begin{equation}
I= \int_{\Sigma} \ell_0(t,x,u)\,d\mu^N(t,x,u).
\end{equation}
Notice that
\begin{align}
&\lambda\iint W(x-y)\,\delta(t-t')\,
d\mu^N(t,x,u)\,d\mu^N(t',y,v) \nonumber \\
&=\frac{\lambda}{N^2}\sum_{i=1}^N \sum_{j=1}^N \iint W(x-y)\,\delta(t-t')\,
d\mu[\tau_i]\,d\mu[\tau_j] \nonumber \\
&=\lambda \int_0^T\frac{1}{N^2} \sum_{i=1}^N \sum_{j=1}^N W(x_i(t)-x_j(t))\,dt = II,
\end{align} where the second equality uses the definition of $\mu[\tau_i]$ and the last equality uses
$\rho_t^N=\tfrac1N\sum_{i=1}^N\delta_{x_i(t)}$.

Similarly, from~\eqref{eq:terminalMeasure} we have
\begin{equation}
    III = \frac1N\sum_{i=1}^N \Psi(x_i(T)) = \int_{\mathcal X}\Psi(x)\,d\nu^N(x).
\end{equation} 
Combining $I$, $II$, and $III$ proves \eqref{eq:JN-measure}.

Fix $v\in C^1([0,T]\times\mathcal X)$. For each agent $i$,
\[
\frac{d}{dt}v\big(t,x_i(t)\big)
=\partial_t v\big(t,x_i(t)\big)+\nabla_x v\big(t,x_i(t)\big)\cdot f\big(x_i(t),u_i(t)\big).
\]
Integrating on $[0,T]$ yields
\[
v\big(T,x_i(T)\big)-v\big(0,x_i(0)\big)
=\int_0^T\!\big(\partial_t v+\nabla_x v\cdot f(x_i,u_i)\big)\,dt.
\]
Average over $i$ and use the occupation–measure identities~\eqref{eq:runningMeasure} and~\eqref{eq:terminalMeasure} we prove~\eqref{eq:weakN}.
\end{proof}
The occupation measures $(\mu^N,\nu^N)$ encode the same distributional information as the empirical measures $\rho_t^N$ and $\rho_T^N$ in the formulation~\eqref{eq:NagentMF}. In particular, $d\mu^N(t,x,u) = dt\,\mu_t^N(dx,du)$, where $\mu_t^N = 1/N\sum_{i=1}^N \delta_{(x_i(t),u_i(t))}$. The state marginal is
\[
\rho_t^N = \pi_x\# \mu_t^N 
= \tfrac{1}{N}\sum_{i=1}^N \delta_{x_i(t)},
\]
{where $\pi_x$ denotes the projection onto the state component,} and the terminal empirical distribution $\rho_T^N$ coincides with $\nu^N$. {Thus the occupation–measure formulation provides a distribution-level
representation of the $N$-agent problem within a convex framework.} 

\subsection{Optimization Problem over Measures}
{Motivated by Theorem~\ref{Thm:objJN}, 
we reformulate the $N$-agent cost as an optimization problem over occupation measures}:
\begin{align}\label{mainObj}
\min_{(\mu,\nu)\in\Delta} & 
J(\mu,\nu) :=
\int_{\Sigma} \ell_0(t,x,u)\,d\mu(t,x,u) \nonumber\\
& + \lambda \iint
W(x-y)\,\delta(t-t')\,d\mu(t,x,u)\,d\mu(t',y,v) \nonumber\\
& + \int_{\mathcal X} \Psi(x)\,d\nu(x), \tag{P}
\end{align}
where the feasible set $\Delta$ is
\begin{align}\label{eq:feasibleSet}
\Delta &:= \Big\{ (\mu,\nu)\in \mathcal M_+(\Sigma,T)\times\mathcal M_+(\mathcal X,1) :\;\nonumber\\
&\int_{\Sigma} \Big(\partial_t v+\nabla_x v\cdot f(x,u)\Big)\,d\mu(t,x,u) \nonumber\\
&= \int_{\mathcal X} v(T,x)\,d\nu(x) - v(0,x_0),
\, \forall v\in C^1([0,T]\times\mathcal X) \Big\}.
\end{align}
{This problem serves as a relaxation of the finite-agent formulation, a viewpoint that is standard in measure-based relaxations of nonconvex problems \cite{bonnans2023large}.}
In the rest of the paper we will focus on problem~\eqref{mainObj}, which is an \emph{infinite-dimensional} optimization problem over measures. In the next section we establish key properties of this problem, including convexity of the feasible set and objective, and the first variation of the objective functional. These properties naturally motivate the use of the FW method, which is projection-free and can directly handle infinite-dimensional problems.

\section{Properties of the Measure Optimization Problem}\label{sec:properties}
\begin{Theorem}[Convexity and compactness]\label{thm:conv} The following assertions about problem~\eqref{mainObj} hold:
\begin{enumerate}
  \item[(a)] If the interaction kernel $W:\mathbb{R}^{n_x}\to\mathbb{R}$ is positive
  semidefinite, then $J(\mu,\nu)$ is convex on
  $\mathcal M_+(\Sigma,T)\times\mathcal M_+(\mathcal X,1)$.
  \item[(b)] The feasible region $\Delta$ is a convex and weak$^*$ compact subset of
  $\mathcal M_+(\Sigma,T)\times\mathcal M_+(\mathcal X,1)$.
\end{enumerate}
\end{Theorem}

{\emph{Note.}
Here, weak$^*$ compactness is understood with respect to the weak$^*$ topology on measures, i.e.,
a sequence $\{\mu_k\}$ converges weak$^*$ to $\mu$ if
$\int \phi \, d\mu_k \to \int \phi \, d\mu$ for all $\phi \in C_b(\Sigma)$
(and similarly for measures on $\mathcal X$).}

\begin{proof}
    The first and last terms in $J(\mu,\nu)$ are linear, so it suffices to prove
convexity of the interaction term
\[
F(\mu):=\iint W(x-y)\,\delta(t-t')\,d\mu(t,x,u)\,d\mu(t',y,v).
\] Let $\mu_1,\mu_2\in\mathcal M_+(\Sigma,T)$ and $\alpha\in[0,1]$, A direct expansion gives
    \begin{align*}
        &F(\alpha\mu_1+(1-\alpha)\mu_2) - (\alpha F(\mu_1)+(1-\alpha)F(\mu_2)) \\
        &\quad= -\alpha(1-\alpha)\iint W\,\delta\,d(\mu_1-\mu_2)\,d(\mu_1-\mu_2) \leq 0,
    \end{align*} where the inequality follows from $W$ being positive semidefinite and $\delta$ being nonnegative. Thus $F$ is convex, proving (a). For (b), note that $\mathcal M_+(\Sigma,T)\times\mathcal M_+(\mathcal X,1)$ is convex, and the Liouville constraint in~\eqref{eq:feasibleSet} is linear in $(\mu,\nu)$. Hence $\Delta$ is convex. Since $\mathcal M_+(\Sigma,T)$ and $\mathcal M_+(\mathcal X,1)$ are weak$^*$ compact and the Liouville map in the constraint is weak$^*$ continuous, therefore $\Delta$ is weak$^*$ compact.
\end{proof}

Theorem~\ref{thm:conv} immediately implies the existence of solutions to problem~\eqref{mainObj}.
\begin{Corollary}[Solution Existence]
Problem~\eqref{mainObj} admits at least one optimal solution 
$(\mu^*,\nu^*) \in \Delta$.
\end{Corollary}

\begin{Theorem}[First Variation of Objective]\label{thm:firstVariation}
Fix $(\mu,\nu)\in\Delta$. For any signed measures $\delta\mu\in\mathcal M(\Sigma)$ and $\delta\nu\in\mathcal M(\mathcal X)$ with
$\delta\mu(\Sigma)=0$, $\delta\nu(\mathcal X)=0$ that satisfy the
Liouville constraint $\forall v\in C^1([0,T]\times\mathcal X)$
\[
\int_\Sigma (\partial_t v+\nabla_x v\!\cdot\! f)\,d\delta\mu
=\int_{\mathcal X} v(T,x)\,d\delta\nu(x),
\]
the directional derivative of $J$ at $(\mu,\nu)$ along $(\delta\mu,\delta\nu)$
exists and is given by
\begin{align}\label{eq:firstVariation}
\delta J(\mu,\nu;\delta\mu,\delta\nu)&:=\left.\frac{d}{d\varepsilon}\,J(\mu+\varepsilon\delta\mu,\nu+\varepsilon\delta\nu)
\right|_{\varepsilon=0} \nonumber\\
&=:\langle g_{\mu},\delta\mu\rangle + \langle \Psi,\delta\nu\rangle,
\end{align}
where $\langle g_{\mu},\delta\mu\rangle:=\int_\Sigma g_{\mu}(t,x,u)\,d\delta\mu(t,x,u)$ and
$\langle \Psi,\delta\nu\rangle:=\int_{\mathcal X}\Psi(x)\,d\delta\nu(x)$, with
\begin{equation}\label{eq:gmu}
g_{\mu}(t,x,u)=\ell_0(t,x,u)+2\lambda\!\!\int
W(x-y)\,\delta(t-t')\,d\mu(t',y,v).
\end{equation}
\end{Theorem}
\begin{proof}
Linearity of the first and last terms in $J(\mu,\nu)$ gives 
$\langle \ell_0,\delta\mu\rangle$ and $\langle \Psi,\delta\nu\rangle$. 
For the interaction term 
\[
F(\mu):=\iint W(x-y)\,\delta(t-t')\,d\mu(t,x,u)\,d\mu(t',y,v),
\]
we expand
\begin{align*}
&\lim_{\varepsilon\to 0}\frac{F(\mu+\varepsilon\delta\mu)-F(\mu)}{\varepsilon} \\
&\quad= 2 \iint W\,\delta\,d\mu\,d\delta\mu 
  + \lim_{\varepsilon\to 0}\varepsilon \iint W\,\delta\,d\delta\mu\,d\delta\mu \\
&\quad= 2 \iint W\,\delta\,d\mu\,d\delta\mu.
\end{align*}
Combining terms proves~\eqref{eq:firstVariation}.
\end{proof}

These properties confirm that problem~\eqref{mainObj} is a well-posed convex optimization problem in an infinite-dimensional space. Moreover, the first variation of $J$ is linear in $(\delta\mu,\delta\nu)$, which makes it the natural first-order object for constructing linearized subproblems. In the next section, we exploit this structure to design an FW type algorithm over measure spaces.

\section{Frank–Wolfe Methods for Optimization over Measures}\label{sec:fw_methods}

In this section, we introduce FW methods for problem~\eqref{mainObj}. To motivate our approach, recall the FW recursion (also known as the \emph{conditional gradient} method~\cite{2015bub}) in finite-dimensional Euclidean spaces. When minimizing a smooth function \( f: \mathbb{R}^d \to \mathbb{R} \) over a compact convex set \( Z \subset \mathbb{R}^d \), the FW recursion is given by
\begin{equation} \label{eq:FW}
    y_{k+1} = (1 - \gamma_k)y_k + \gamma_k s_k, \quad s_k := \argmin_{s \in Z} \nabla f(y_k)^\top s,
\end{equation}
where \( \gamma_k \in (0,1] \) is a step size. This method is ``projection-free'' and primal in that the iterates \( \{y_k\} \) are always feasible; whether the method is efficient depends on how easily the subproblem can be solved.

To extend FW to problem~\eqref{mainObj}, we linearize the objective using the first variation $\delta J$ at $(\mu,\nu)\in\Delta$.
\[
J(\mu,\nu) \approx J(\mu_k,\nu_k) + \delta J(\mu_k,\nu_k;\,\mu-\mu_k,\nu-\nu_k).
\]
This motivates the following recursion over measures:
\begin{align}\label{eq:FWrec}
&(\mu_{k+1},\nu_{k+1}) = (1-\gamma_k)(\mu_k,\nu_k) + \gamma_k (\tilde{\mu}_k,\tilde{\nu}_k), \nonumber\\
&(\tilde{\mu}_k,\tilde{\nu}_k) := \argmin_{(\tilde{\mu},\tilde{\nu})\in \Delta} \, \langle g_{\mu_k},\tilde{\mu}-\mu_k\rangle + \langle \Psi,\tilde{\nu}-\nu_k\rangle,
\end{align}
where $g_{\mu_k}$ is the function defined in~\eqref{eq:gmu}. Since the subproblem is linear over the convex feasible set $\Delta$, its minimizers can be chosen among the occupation measures generated by admissible trajectories. The next theorem makes this explicit.

\begin{Theorem}[Solution to Frank–Wolfe Subproblem]\label{thm:lmo-single}
For fixed $(\mu_k,\nu_k)$, let $\tau_k=(x_k(\cdot),u_k(\cdot))$ be an admissible trajectory satisfying
\begin{equation*}
    \tau_k \in \argmin_{\substack{\dot x=f(x,u)\;\\ x(0)=x_0}}
    \ \int_0^T g_{\mu_k}(t,x(t),u(t))\,dt + \Psi(x(T)).
\end{equation*}
Then the occupation--measure pair $(\mu[\tau_k],\nu[\tau_k])$ is a minimizer of the FW subproblem, i.e.,
\[
(\mu[\tau_k],\nu[\tau_k]) \in 
\argmin_{(\mu,\nu)\in\Delta}\ \langle g_{\mu_k},\mu\rangle+\langle\Psi,\nu\rangle.
\]
\end{Theorem}

\proofsketch
Let $(\mu,\nu)\in\Delta$. By the probabilistic representation
\cite[Thm.~8.2.1]{ambrosio2005gradient}, {the occupation measures can be disintegrated into a probability measure $\Pi$ over admissible state trajectories. For the purpose of this sketch, we assume these decomposed paths admit classical control representations $\tau=(x(\cdot),u(\cdot))$ with $x(0)=x_0$ and $\dot x=f(x,u)$. A fully rigorous proof requires addressing relaxed controls and approximation arguments \cite{vinter2010optimal}, which are omitted here for brevity. Under this assumption, we have for all $\varphi\in C(\Sigma)$ and $\psi\in C(\mathcal X)$, }
\begin{align*}
\int_\Sigma \varphi\,d\mu&=\int\!\Big(\int_0^T \varphi(t,x(t),u(t))\,dt\Big)d\Pi(\tau),\\
\int_{\mathcal X}\psi\,d\nu&=\int \psi(x(T))\,d\Pi(\tau).
\end{align*}
Taking $\varphi=g_{\mu_k}$ and $\psi=\Psi$ gives
\begin{align*}
&\langle g_{\mu_k},\mu\rangle+\langle\Psi,\nu\rangle \\
&\quad= \int \Big(\int_0^T g_{\mu_k}(t,x(t),u(t))\,dt+\Psi(x(T))\Big)\,d\Pi(\tau) \\
&\quad \ge \min_{\substack{x(0)=x_0\\ \dot x=f(x,u)\ }}
\Big\{\int_0^T g_{\mu_k}(t,x(t),u(t))\,dt+\Psi(x(T))\Big\}.
\end{align*}
By definition of $\tau_k$, the minimum is attained at $\tau_k$; choosing $\mu=\mu[\tau_k]$ and $\nu=\nu[\tau_k]$
yields equality.
\endproof

By Theorem~\ref{thm:lmo-single}, the FW subproblem reduces to selecting the occupation measures of a single admissible trajectory that minimizes the linear objective. Equivalently, the subproblem becomes a \emph{classical optimal control problem} with stage cost $g_{\mu_k}$ and terminal cost $\Psi$. The FW recursion in~\eqref{eq:FWrec} takes the form
\begin{align}\label{FWRecursion}
&(\mu_{k+1},\nu_{k+1})
= (1-\gamma_k)(\mu_k,\nu_k) + \gamma_k\,(\mu[\tau_k],\nu[\tau_k]), \nonumber\\
&\tau_k \in \argmin_{\substack{x(0)=x_0\\ \dot x=f(x,u)\ }}
\ \int_0^T g_{\mu_k}(t,x(t),u(t))\,dt + \Psi(x(T)). 
\end{align}
The approach implied by~\eqref{FWRecursion} thus solves the infinite-dimensional optimization problem by iteratively solving classical optimal control problems and aggregating their occupation measures in Algorithm~\ref{alg:FW}.


\begin{algorithm}[H]
\caption{Frank--Wolfe over Occupation Measures}
\label{alg:FW}
\small
\begin{algorithmic}[1]
\State \textbf{Input:} Initial measures $(\mu_0,\nu_0)$, step sizes $\{\gamma_k\}$
\For{$k=0,1,2,\dots, K$}
    \State Solve the optimal control problem
    \[
      \tau_k \in \argmin_{\substack{x(0)=x_0\\ \dot x=f(x,u)}}
      \ \int_0^T g_{\mu_k}(t,x(t),u(t))\,dt + \Psi(x(T)),
    \] \hfill (with $g_{\mu_k}$ defined in~\eqref{eq:gmu})
    \State Form $(\mu[\tau_k],\nu[\tau_k])$.
    \State 
      $(\mu_{k+1},\nu_{k+1}) = (1-\gamma_k)(\mu_k,\nu_k) + \gamma_k\,(\mu[\tau_k],\nu[\tau_k])$.
\EndFor
\end{algorithmic}
\end{algorithm}
Similar to classical FW methods, with step-sizes $\gamma_k=2/(k+2)$, Algorithm~\ref{alg:FW} guarantees an $O(1/k)$ convergence rate in objective value. A natural extension is the \emph{fully-corrective} variant, which maintains all previously generated trajectories and re-optimizes their weights instead of forming a simple convex combination. This often improves practical performance. We refer to~\cite{yu2025deterministic,yu2025frank} for detailed descriptions and convergence analysis of FW methods in the measure-optimization setting.

One important remark is that although the optimal control problem in Step~3 of Algorithm~\ref{alg:FW} depends on the current iterate $\mu_k$, it remains tractable. Since $\mu_k$ can be written as $\sum_{i=0}^k \alpha_i \mu[\tau_i]$ with trajectories $\tau_i=(x_i,u_i)$, the function $g_{\mu_k}$ has a closed-form expression, and the subproblem reduces to
{\small
\begin{align*}
 \min_{\substack{x(0)=x_0\\ \dot x=f(x,u)}} \;
  &\int_0^T \left( \ell_0(t,x,u)
 + 2\lambda \sum_{i=0}^k \alpha_i W(x-x_i)\right) \,dt + \Psi(x(T)).
\end{align*} }
This is a standard optimal control problem, for which a broad range of mature solution methods and solvers are available, including direct transcription, shooting, and model predictive control–based approaches. Such problems are well studied and supported by efficient software packages, ensuring that the subproblem is not a computational bottleneck. In the numerical experiments, a lightweight gradient-based solver is employed to illustrate that the FW subproblem can be addressed effectively in practice.
{\begin{remark}
The Frank--Wolfe algorithm yields a measure-valued solution
$\mu_K = \sum_{k=0}^{K} \alpha_k \mu[\tau_k]$, which characterizes an optimal occupation measure at the planning level and, for model simplicity, assumes identical initial conditions for all UAVs.
This paper focuses on the formulation and optimization of the mean-field problem, and a gap therefore remains between the planning-level solution and its execution on a finite swarm of UAVs.
In future work, we will employ sampling-based approaches (e.g., Monte Carlo methods) to jointly address the initial state distribution and the realization of the planning-level solution, and to further investigate systematic implementation mechanisms for finite swarms.
\end{remark}}

\section{Numerical Illustration}\label{sec:numerical}

We evaluate the method on a UAV–swarm path–planning task from \(x_0\) to \(x_g\) over horizon \(T\). Each agent follows single–integrator dynamics
\[
\dot x_i(t)=u_i(t),\qquad x_i(0)=x_0,\quad i=1,\ldots,N .
\]

The objective is to minimize the average cost over all agents,
\begin{equation}
\label{eq:sim_objective}
\begin{aligned}
J_N=&\frac{1}{N}\sum_{i=1}^N\Bigg[\int_{0}^{T}\Big(\tfrac{\alpha}{2}\|u_i(t)\|^2 + V_{\mathrm{obs}}(x_i(t))\Big) \,dt \\
&+\int_{0}^{T}\frac{\gamma}{2}\!\!\sum_{j\neq i} W(\|x_i(t)-x_j(t)\|)\,dt \\
&+\frac{\lambda_\Psi}{2}\|x_i(T)-x_g\|^2
\Bigg].
\end{aligned}
\end{equation} 
Note that in this problem the interaction term $W$ couples the agents, so the objective is not simply the sum of individual costs. Reformulated as in Problem~\eqref{mainObj}, this becomes
{\small
\begin{align}\label{eq:sim_measure}
J &= \int \Big(\tfrac{\alpha}{2}\|u\|^2 + V_{\mathrm{obs}}(x)\Big)\,d\mu 
+ \frac{\lambda_\Psi}{2}\int_{\mathcal X}\|x-x_g\|^2\,d\nu(x) \nonumber\\
&\quad + \frac{\gamma}{2}\iint 
W(x-y)\,\delta(t-t')\,d\mu(t,x,u)\,d\mu(t',y,v).
\end{align}}

Repulsion uses an isotropic Gaussian kernel
\(
W(r)=\kappa\exp(-r^{2}/(2\sigma^{2}))
\),
with width \(\sigma\) (sensing radius); \(\kappa\) is set so that repulsion at spacing \(d_{\min}\) balances the terminal pull.
Obstacles are modeled by a smooth potential \(V_{\mathrm{obs}}\).

\subsection{2D Swarm with Repulsion and a Single Obstacle: Verification against PDE}
We consider a scenario in which agents must avoid an obstacle while simultaneously maintaining separation through repulsive interaction. Specifically, a circular obstacle of radius $r_{\mathrm{obs}}=0.8$ is centered at $(2.5,1.5)$ with a margin of $0.2$, and the obstacle potential is modeled by a quadratic penalty of weight $\beta=10^3$. The inter--agent coupling is introduced via the repulsive kernel $W$, parameterized by $\lambda_W=1.5$ and $\sigma_W=0.15$. All other parameters remain the same as in the previous setting: horizon $T=3$, control weight $\alpha=0.1$, terminal penalty $\lambda_\Psi=30$, goal position $x_g=(5,3)$, and initial distribution $\rho_0=\mathcal N((0.12,0.12),\sigma_0^2I)$ with $\sigma_0=0.07$. In this configuration, the optimal swarm distribution must balance goal--seeking, control effort, obstacle avoidance, and mutual repulsion. 

We implement Algorithm~\ref{alg:FW} with step sizes $\gamma_k = 2/(k+2)$ over horizon $T=3$, running $100$ outer iterations. At each iteration, the linear minimization oracle requires solving an optimal control subproblem. In our setting, this subproblem reduces to a quadratic program (QP) with linear dynamics and quadratic cost, a standard form in optimal control that can be handled efficiently using existing solvers. We employ a lightweight gradient-based routine to solve this QP in practice, which is sufficient given the moderate problem dimension and short horizon. 
\begin{figure}[t]
    \centering
    \includegraphics[width=\linewidth]{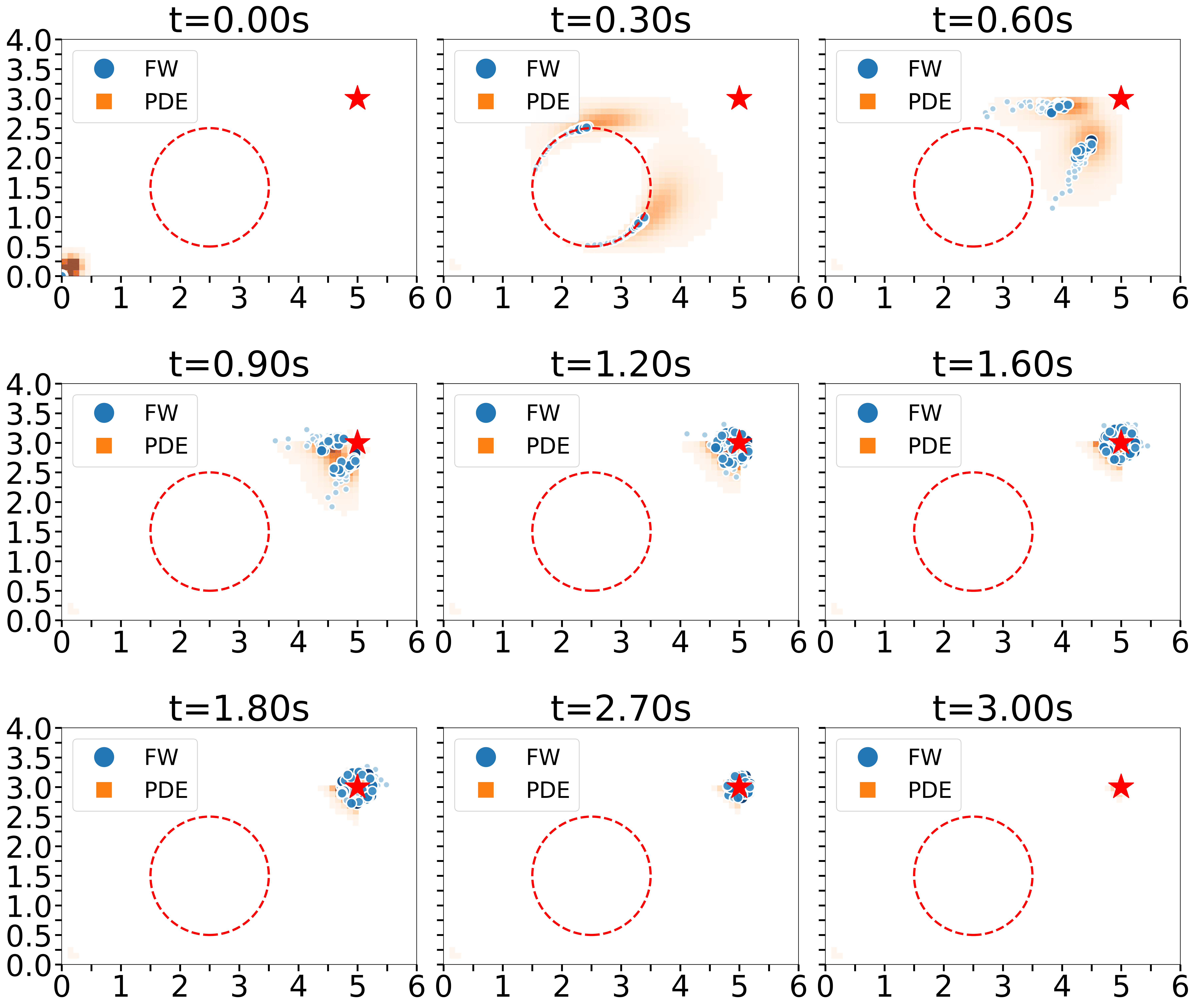}
    \caption{\footnotesize 2D swarm with repulsion and one obstacle: FW (blue points) vs. PDE (orange heatmap).}
    \label{fig:obstacle_compare_2d}
\end{figure}
\begin{figure}[t]
    \centering
    \includegraphics[width=0.45\textwidth]{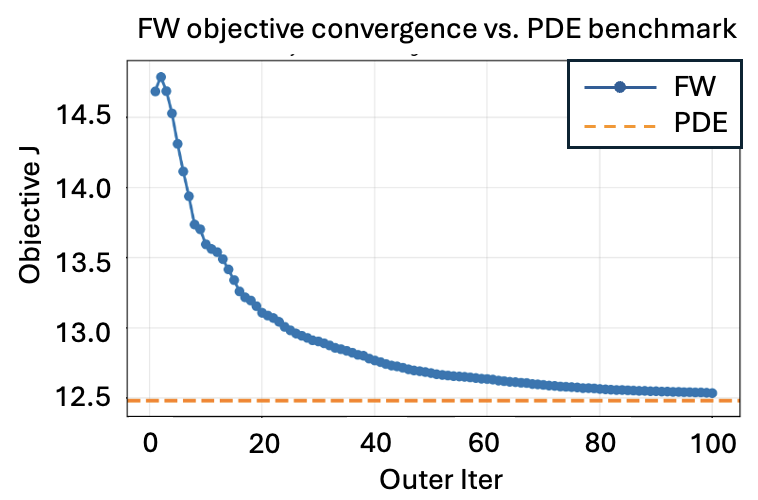}
    \caption{\footnotesize Objective \(J\): FW (blue) approaches PDE (orange).}
    \label{fig:obj_converge}
\end{figure}

Figure~\ref{fig:obstacle_compare_2d} compares the FW solution, visualized as time slices of the occupation measures $\mu_k$ (blue points, with darker points indicating larger weights), with the PDE benchmark density (orange heatmap) in the presence of a repulsive potential and a single obstacle (dashed circle). Both approaches yield similar swarm evolution and terminal distributions, with the agents coordinating to avoid the obstacle and converge to the target (red star). 
Figure~\ref{fig:obj_converge} shows the FW objective decreasing over iterations and approaching the PDE benchmark value. This demonstrates convergence in objective value and validates that the FW scheme faithfully approximates the PDE solution in this more challenging setting.

\subsection{3D Swarm with Repulsion and Multiple Obstacles: Scalability}
In this part, we solve a 3D swarm control problem with ten obstacles and repulsive interactions. For this setting, the computational cost of PDE-based solvers grows explosively, while the FW approach remains much more scalable since it does not rely on discretizing the state space. The parameters are:
$ x_0=(0,0,0),\quad x_g=(5,3,2),\quad T=3,\; N_t=150; \alpha=0.1,\; \lambda_\Psi=30,\; \lambda_W=25,\; \beta=10^3.$
We model obstacles via a smooth potential
\(V_{\mathrm{obs}}:\mathbb{R}^3\to\mathbb{R}_+\) obtained as the superposition
of radial barrier functions around ten randomly generated spherical
obstacles with centers \(c_i\) and safety radii \(R_i\), namely
\[
  V_{\mathrm{obs}}(x)
  = \beta \sum_{i=1}^{10}
    \tau^2 \log^2\!\bigl(1 + e^{(R_i - \|x-c_i\|)/\tau}\bigr),
\]
where \(\beta>0\) and \(\tau>0\) are fixed weights.

Algorithm settings: \(K=50\) FW outer iterations.


In three-dimensional settings, grid-based PDE solvers typically require fine spatial discretization, leading to substantial computational and memory demands. In contrast, the proposed FW framework operates directly in the occupation-measure space and does not rely on state-space gridding.
This enables us to extend the method to the 3D swarm setting with ten obstacles and repulsive interactions at a practical computational cost.
\begin{figure}[t] 
\centering 
\includegraphics[width=\linewidth]{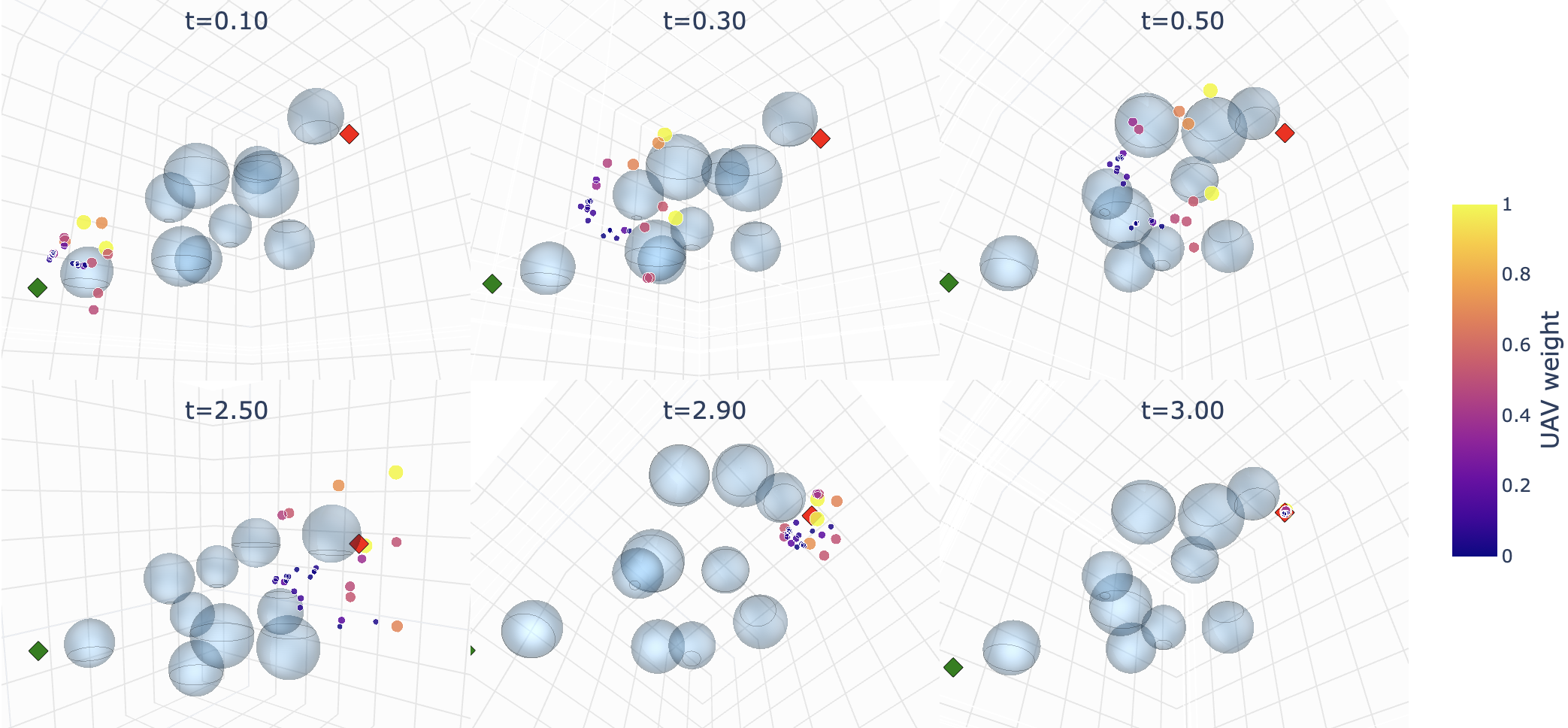} 
\caption{3D swarm with repulsion and 10 obstacles.} 
\label{fig:obstacle_compare_3d} 
\end{figure}

Figure~\ref{fig:obstacle_compare_3d} shows time slices of the UAV distribution obtained from the FW solution. Although the 3D perspective may visually suggest that some UAVs overlap with obstacles, they are in fact outside the obstacle boundaries.
The swarm successfully avoids obstacles, maintains repulsive separation, and gradually converges to the target region. 
The entire 3D instance, including interactions and ten obstacles, is solved in about 11 minutes on a standard workstation, demonstrating the practicality of FW.


\section{Conclusion}\label{sec:conclusion}
We have presented an optimization-based framework for mean-field control that lifts the problem into the space of occupation measures, yielding a convex formulation amenable to Frank–Wolfe methods. By applying the Frank–Wolfe algorithm in this convex measure space, the problem reduces to a sequence of tractable optimal control subproblems, while inheriting provable convergence guarantees. Numerical experiments demonstrate consistency with analytic and PDE-based references in two dimensions and show that the method remains effective in three-dimensional scenarios with multiple obstacles. A full 3D instance with ten obstacles is solved in minutes on a standard workstation, illustrating the practical applicability of the approach. These results suggest that convex optimization in measure spaces provides a viable computational framework for large-scale swarm control problems.

Despite these advances, several challenges remain. Solving the reduced optimal control steps still incurs non-negligible computational cost, which limits real-time applicability. Future research directions include integrating the framework with model predictive control to enable real-time implementations, extending the formulation to stochastic and uncertain dynamics, and investigating distributed realizations suitable for hardware deployment.
\bibliographystyle{ieeetr}
\bibliography{ref}

\end{document}